\documentclass[12pt, letterpaper]{amsart}
\usepackage[margin=1in]{geometry}
\usepackage{color,url}
\usepackage{tikz,enumerate}
\usetikzlibrary{cd}

\usepackage{amssymb}

\title{Price's Law for the massless Dirac--Coulomb system
}
\author[D. Baskin]{Dean Baskin}
\email{dbaskin@math.tamu.edu}
\address{Department of Mathematics, Texas A\&M University \\ Mailstop
  3368 \\ College Station, TX 77843}
\author[J. Gell-Redman]{Jesse Gell-Redman}
\email{jgell@unimelb.edu.au}
\address{School of Mathematics and Statistics, University of Melbourne \\ VIC 3010 \\ Australia}
\author[J.L. Marzuola]{Jeremy L. Marzuola}
\email{marzuola@math.unc.edu}
\address{Department of Mathematics, UNC-Chapel Hill \\ CB\#3250
  Phillips Hall \\ Chapel Hill, NC 27599}
  
  \thanks{The authors are grateful to Peter Hintz and Andras Vasy for
    helpful conversations at various points regarding the relationship
    between radiation fields and Price's Law.
D.B.\ was partially supported by NSF CAREER grant
DMS-1654056, and by the Mathematical Sciences
Research Institute in Berkeley, California, supported by the NSF under Grant No. DMS-1928930.
 J.G.R.\ acknowledges support from Australian Research Council through
 the Discovery Project grants DP180100589 and DP210103242.
 J.L.M.\ acknowledges support from NSF grant DMS-1909035.
}

\newcommand{\charge}{\mathbf{Z}}

\newcommand{\cf}{\mathrm{cf}}
\newcommand{\mf}{\mathrm{mf}}
\newcommand{\scri}{\mathcal{I}}

\newcommand{\ff}{{\rm ff}}
\newcommand{\tf}{{\rm tf}}

\newcommand{\E}{\mathcal{E}}
\newcommand{\F}{\mathcal{F}}
\newcommand{\DC}{\text{DC}}
\newcommand{\invsq}{\text{IS}}

\renewcommand{\Im}{\operatorname{Im}}

\newcommand{\norm}[1]{{\left\lVert{#1}\right\rVert}}

\newcommand{\abs}[1]{\lvert #1\rvert}

\newcommand{\yyy}{y}

\newcommand{\lap}{\Delta}

\newcommand{\pd}[1][]{\partial_{#1}}
\newcommand{\reals}{\mathbb{R}}
\newcommand{\sphere}{\mathbb{S}}
\newcommand{\CC}{\mathbb{C}}

\newtheorem{theorem}{Theorem}
\newtheorem{lemma}[theorem]{Lemma}

\theoremstyle{definition}

\newtheorem{remark}[theorem]{Remark}

\numberwithin{theorem}{section}

\begin{document}

\begin{abstract}
  We consider the pointwise decay of solutions to wave-type
  equations in two model singular settings.  Our main result is a form
  of Price's law for solutions of the massless Dirac--Coulomb system
  in $(3+1)$-dimensions.  Using identical techniques, we prove a
  similar theorem for the wave equation on Minkowski space with an
  inverse square potential.  One novel feature of these singular
  models is that solutions exhibit two different leading
  decay rates at timelike infinity in two regimes, 
  distinguished by whether the spatial momentum along a curve which
  approaches timelike infinity is zero or non-zero.  An important
  feature of our analysis is that it yields a precise description of
 solutions at the interface of these two regions which comprise the
 whole of timelike infinity.
\end{abstract}

\maketitle

\section{Introduction}
\label{sec:introduction}

We consider the long-time
asymptotics of solutions of the wave equation on backgrounds with
scale invariant singular potentials and provide pointwise
  decay estimates in all asymptotic regimes.
The models we consider here are those of the massless Dirac--Coulomb
equation, and of the wave equation with an inverse
square potential, the former equation was considered by Baskin and
Gell-Redman (with Booth)~\cite{DirCoulRF} and the latter by Baskin
and Marzuola~\cite{coneradfield}.  In both instances, in a departure
from standard formulations of Price's law, we find that solutions
exhibit different leading rates of decay depending on how timelike
infinity is approached.

For the Dirac--Coulomb equation, the microlocal framework
  in prior work~\cite{DirCoulRF} produces regions of timelike infinity at
  which solutions exhibit distinct decay rates.  Indeed, the regions
  of timelike infinity arise as boundary hypersurfaces of an
  appropriate spacetime compactification; that work~\cite{DirCoulRF}
  provides a global description of solutions on a closely related compactification.
This requires a clear understanding of the nature of solutions (namely in
terms of polyhomogeneity, which they possess in certain regions of
spacetime) extracted from \cite{DirCoulRF}.  We
show, moreover, that the closely related description of
  solutions to wave equations in conic setting~\cite{coneradfield}
fits into the same framework and thus
the solutions to the wave equation with inverse square potentials have
the same feature.

We consider first the massless Dirac--Coulomb equation.  Thus, let $\psi : \reals \times \reals^3 \to \CC^4$ a solution to
\begin{align}
  \label{eq:dc}
& \left( \gamma^0 \left( \partial_t + \frac{i Z}{r} \right) + \sum_{j=1}^3 \gamma^j \partial_j \right) \psi = 0, \\
& \psi(0) = \psi_0 \in  (C^{\infty}_{c}(\reals^{3}\setminus\{ 0\}))^4 ,\notag
\end{align}
for $\gamma^{0,1,2,3}$ the gamma (or Dirac) matrices. In
  other words,
$\gamma^\alpha \in  {\rm Mat} (4;\mathbb{C})$ satisfy 
\begin{equation*}
  \gamma^{\alpha}\gamma^{\beta} + \gamma^{\beta}\gamma^{\alpha} = -2
  \eta^{\alpha\beta} \operatorname{Id}_{4},
\end{equation*}
where $\eta^{\alpha\beta}$ are the components of the Minkowski metric,
i.e.,
\begin{equation*}
  \eta^{\alpha\beta} =
  \begin{cases}
    -1 & \alpha =\beta=0 ,\\
    1 & \alpha = \beta \in \{1,2,3\} ,\\
    0 & \alpha \neq \beta.
  \end{cases}
\end{equation*}
Here $\\charge$ is a real constant with $|\charge| < \frac12$; this range allows
the use of the Hardy inequality and ensures the essential
self-adjointness of the underlying Hamiltonian.

We also consider the wave equation on $\reals \times
\reals^{n}$, $n\geq 3$, with an
inverse square potential.  That is, we take $u$ a solution
to
\begin{align}
  \label{eq:inv-square}
  &\pd[t]^{2}u - \lap u + \frac{\digamma}{r^{2}}u = 0,\\
  &(u,\pd[t]u)|_{t=0} = (\phi_{0}, \phi_{1}) \in
    C^{\infty}_{c}(\reals^{n}\setminus\{ 0\})\times
    C^{\infty}_{c}(\reals^{n}\setminus\{0\}), \notag
\end{align}
where $\lap$ is the negative definite Laplacian on $\reals^{n}$ and
\begin{equation*}
  \digamma > -\frac{(n-2)^{2}}{4}.
\end{equation*}
Due to the assumption on $\digamma$, the Hardy inequality
  implies that the form $\langle (-\lap +
  \frac{\digamma}{r^{2}} ) \phi, \phi \rangle$ is positive definite
  for $\phi \in C^\infty_c(\mathbb{R}^3 \setminus \{ 0 \})$.  For those values of $n$ and $\digamma$ for which the
  Hamiltonian fails to be essentially
  self-adjoint, we use the Friedrichs extension.

In three dimensions, our main result is the following:
\begin{theorem}
  \label{metathm}
  Fix $\chi \in C^{\infty}_{c}(\reals^{3}\setminus 0)$.
  For $\psi$ the solution of the Dirac--Coulomb system~\eqref{eq:dc} and
  $u$ the solution of the inverse square wave
  equation~\eqref{eq:inv-square} with $n = 3$, we have, as $t\to \infty$,
  \begin{equation*}
    \norm{\chi(\cdot) \psi(t,\cdot)} \lesssim \langle t \rangle^{-3 - \alpha(\charge)}, \quad
    \norm{\chi(\cdot)u(t, \cdot)}\lesssim \langle t \rangle^{-2-\beta(\digamma)},
  \end{equation*}
where the exponents are given in terms of $\charge$ and
$\digamma$ by 
\begin{equation*}
  \alpha(\charge) = 2 \sqrt{1-\charge^{2}}-2,
\end{equation*}
and
\begin{equation*}
  \beta(\digamma) = \sqrt{1 + 4\digamma} -1
\end{equation*}
provided $\sqrt{1 + 4\digamma}$ is not an odd integer.  If $\sqrt{1 + 4\digamma}$ is an odd integer, then
  $\beta (\digamma)$ is $\sqrt{9
    + 4\digamma}-1$.   

In contrast, for $0 < \gamma < 1$, the solutions obey the estimates
\begin{equation*}
  \abs{\psi(t,\gamma t, \theta)} \lesssim t^{-3-\frac{\alpha(\charge)}{2}} , \quad \abs{u(t,\gamma t,
    \theta)} \lesssim t ^{-2-\frac{\beta(\digamma)}{2}}.
\end{equation*}

Moreover, all decay rates above are sharp for $\charge, \digamma\neq
0$ in the permitted range (see Section \ref{sec:sharp}).
\end{theorem}

We emphasize that the main novelty of Theorem~\ref{metathm} is the
first part; the latter estimates (for $0 < \gamma < 1$) were already
known in our prior works~\cite{coneradfield,DirCoulRF}.  Thus we see
that \emph{the rates of decay in the two regimes differ}; one holds on
bounded spatial sets (here the support of the arbitrary compactly
supported function $\chi$) and the other holds on timelike
trajectories with nonzero spatial momentum. Note that as the parameters $\charge$ and $\digamma$ vary
farther from $0$, the contrast between these two behaviors grows.
For Dirac--Coulomb, the behavior on compact subsets always decays
more slowly than in the $\gamma > 0$ case as $\alpha(\charge) < 0$.
In contrast, for the inverse square wave equation, the sign of $\digamma$ determines in which region the
solution's decay is faster. 

We also note that the results above naturally invite a
  comparison with $t^{-3}$ (for Dirac--Coulomb) and $t^{-2}$ (for
  inverse square) rates.  These particular rates may seem unnatural,
  as Price's law would suggest rates of $t^{-4}$ and $t^{-3}$,
  respectively, for large classes of smooth perturbations of Minkowski
  space when $\charge = 0$ or $\digamma = 0$.  In the context of these
  families of equations, however, the rates predicted by Price's law
  are the outlier and arise from a significant cancellation related to
  Huygens' principle. A brief discussion of this discrepancy (and
  related ones) can be found in Section~\ref{sec:an-asympt-expans}
  below as well as the first author's prior work~\cite{BVW2}.  We
further refer the reader to the works of Tataru~\cite{Tataru} and Morgan
\cite{Morgan20} for an overview of the $t^{-3}$ decay rate in the
asymptotically Minkowski setting.

\begin{remark}
  Here we work with the explicit operators of the form
  $\pd[t]^{2} - \lap + \frac{\digamma}{r^{2}}$, but it is natural to
  ask about equations of the form
  $\Box_g + \frac{\digamma}{r^{2}} + V$ for $g$ an asymptotically flat
  metric and $V$ a non-singular, lower order perturbation.  The
  effects of these perturbations would most strongly manifest at low
  frequencies, and hence we conjecture that the decay rates we prove
  here would still be the expected behavior.  Proving this would
  require adapting the second microlocalization tools of Vasy
  \cite{vasy:lag1,vasy:lag2,vasy:lag3} as implemented in the wave
  decay setting in \cite{Hintz_PL} to the conic setting, and is
  an important topic for future work.
\end{remark}

As observed in the previous works of the first author, Vasy, and
Wunsch~\cite{BVW1,BVW2}, the pointwise decay rates of waves are
in some contexts governed by the
asymptotics of the associated radiation field.  For wave equations on both Minkowski and a
class of asymptotically Minkowski spaces, the corresponding decay
estimate in a neighborhood of null infinity is often what
  is referred to when one discusses Price's Law.  The uniformity of
the estimate across timelike infinity there can be interpreted as a
consequence of the smoothness of the operator on the radial
compactification.
In contrast, the singular potentials considered here exhibit a
different rate of decay at precisely the region of timelike infinity
corresponding to the location of the singularity.

\medskip

There is a plentitude of recent results on Price's Law in the setting
of stationary spacetimes.  We refer to Schlag \cite{Sch21} for a more
thorough overview than we provide here.  Such work goes back at least
to the analysis using spherical harmonic decompositions in
Donninger--Schlag--Soffer \cite{DSS11} and Schlag--Soffer--Staubach
\cite{SSS10-I,SSS10-II}, followed by the work of Tataru \cite{Tataru}
that relied upon Fourier localization techniques.  The methods of
Tataru have now been generalized to a large number of
settings~\cite{MTT12,MTT17}; in particular, the role of regularity of
the metric at infinity was recently studied in the work of
Morgan~\cite{Morgan20} and Morgan--Wunsch~\cite{MorganWunsch21}.  The
work of  Looi~\cite{Looi:2022} considers a
closely related problem to those considered by Morgan and
  Morgan--Wunsch.  The results of Hintz~\cite{Hintz_PL} also employ
microlocal techniques and prove a sharp decay rate for solutions of
the wave equation with a smooth decaying potential, albeit in a
substantially more geometrically complex setting.  Hintz's
  work requires that the potential decay at least as fast as
  $r^{-3}$, i.e., strictly faster than the decay assumed in our
paper.  A pointwise decay estimate for the wave equation with
slowly decaying potential in one spatial dimension was proved by Donninger--Schlag in
\cite{DS10}, though their result has regular potentials at $r=0$ that
decay slightly better than $|r|^{-2}$ as
  $|r| \to \infty$.
The result we prove here for
inverse square potentials is already in the literature (with
substantially different proofs). We refer the reader
  to the recent work of Gajic~\cite{Gajic}, which obtains the same
bimodal decay estimate for the inverse square potential as well as its
analogues on black hole backgrounds.  Gajic's result builds upon the
study of late time asymptotics for the wave equation beginning in the
works \cite{AAG1,AAG2,AAG3}.  See also a forthcoming work of Van de
Moortel and Gajic cited in \cite{Gajic} addressing the Minkowski
setting. However, the authors are aware of no such
  results for the massless Dirac--Coulomb equation.

For nonlinear applications, dispersive estimates for wave and
Schr\"odinger equations have been studied quite broadly.  For
instance, the wave equation with an inverse square potential has been
studied in \cite{BPSS,PSS,MZZ}. The Schr\"odinger equation with an
inverse square potential has been studied in \cite{Duy,KMVZZ1,KMVZZ2}
among many others.  In particular, the distinction between
$\digamma > 0$ and $\digamma < 0$ arises and there is actually a
threshold $-1/4 + 1/25 < \digamma^* < 0$ for which local
well-posedness theory holds for critical semilinear equations with
such potentials.  In the long run, we hope that tighter control on
pointwise behavior of waves will give insight into further nonlinear
applications, see \cite{KMVZ,AM} for the most recent literature
discussion on this problem to date.

The rest of the paper proceeds as follows.  In Section~\ref{S2}, we recall the relevant compactification of our
spacetimes and discuss the necessary preliminary analytic
  results.  Section~\ref{sec:an-asympt-expans} recalls the our main
  results from prior work and proves the constituent results of the
  main theorem (including the $n$-dimensional version for the wave
  equation with an inverse square potential).
Finally, in Appendix \ref{sec:reson-calc}, we provide the resonance
calculations for the case of inverse square potentials 
  providing the explicit decay rates in
  Section~\ref{sec:an-asympt-expans}, the analogous resonances for Dirac-Coulomb already being
  available from \cite{DirCoulRF}.

\section{Compactifications and partial polyhomogeneity}
\label{S2}

We treat both the Dirac--Coulomb equation and the wave equation with
inverse square potential as if they were smooth operators on a
manifold with a conic singularity.  In particular, as $\{0\} \subset
\reals^{n}$ is the singular locus of the potential, we view
$\reals^{n}\setminus\{0\}$ as the interior of the
infinite Riemannian cone over the sphere $\mathbb{S}^{n-1}$.

We now construct a compactification of the overall
  spacetime
  $\mathbb{R}_t \times (\mathbb{R}^n \setminus \{ 0 \}) =
  \mathbb{R}_t \times (0, \infty)_r \times \mathbb{S}^{n-1}$,
  treating it exactly as we would the product of a real time axis
  $\mathbb{R}_t$ with an arbitrary infinite cone $(0,
  \infty)_r \times Z$.  We start by compactifying
$\reals_{t}\times (0,\infty)_{r}$ (and thus
$\reals_{t}\times (0,\infty)_{r} \times Z$) by stereographic
projection to a quarter-sphere $\sphere^{2}_{++}$ as depicted in
Figure~\ref{fig:1-d-compact}.  In other words, the map
$\reals_{t}\times (0,\infty)_{r} \to \sphere^{2} \subset \reals^{3}$
given by
\begin{equation*}
  (t,r) \mapsto \frac{(t,r,1)}{\sqrt{1+t^{2}+r^{2}}}
\end{equation*}
sends $\reals_{t}\times (0,\infty)_{r}$ to the interior of the quarter-sphere given by
\begin{equation*}
  \sphere^{2}_{++} = \{ (z_{1}, z_{2}, z_{3}) \in \sphere^{2} \subset
  \reals^{3} \mid z_{2} \geq 0, \ z_{3} \geq 0\}.
\end{equation*}
The quarter-sphere $\sphere^{2}_{++}$ is a manifold with corners and
has two boundary hypersurfaces defined by the boundary defining
functions $z_{2}$ and $z_{3}$.  We let $\cf$ (or the \emph{conic
  face}) be the hypersurface connecting NP to SP defined by
$\{z_{2}=0\}$, where
\begin{equation*}
  z_{2}= \frac{r}{\sqrt{1+t^{2}+r^{2}}},
\end{equation*}
and we let $\mf$ (or the \emph{main face}) be the face defined by
$\{z_{3}=0\}$, where, over the interior,
\begin{equation*}
  z_{3} = \frac{1}{\sqrt{1+t^{2}+r^{2}}},
\end{equation*}
thus $z_3 \to 0$ along rays where either $t$ or $r$ goes to infinity.
We now let $M = \sphere^{2}_{++}\times \sphere^{n-1}$ and identify the
interior of $M$ with $\reals_{t}\times (\reals^{n}\setminus\{0\})$ by
introducing spherical coordinates in the spatial factor and using the
above identification to map $(t,r)$ to the interior of $\sphere^{2}_{++}$.

Equivalently, we can compactify $\mathbb{R}_t \times (0, \infty)_r$
to a half ball $D_+ = \{ w = (w_1, w_2) \in \mathbb{R}^2: |w| \le 1, w_2\ge
0 \}$ via the mapping
$$
(t, r) \mapsto \frac{(t, r)}{1 + \sqrt{1 + t^2 + r^2}}.
$$
We leave it to the interested reader to show that these
compactifications are equivalent, in the sense that they induce
diffeomorphisms of manifolds with corners.  In particular, the
boundary hypersurface (bhs)
$w_2 = 0$ corresponds to $z_2 = 0$ and the bhs $|w| = 1$ corresponds to $z_3 = 0$.

\begin{figure}
  \centering
      \def\svgwidth{100mm}
   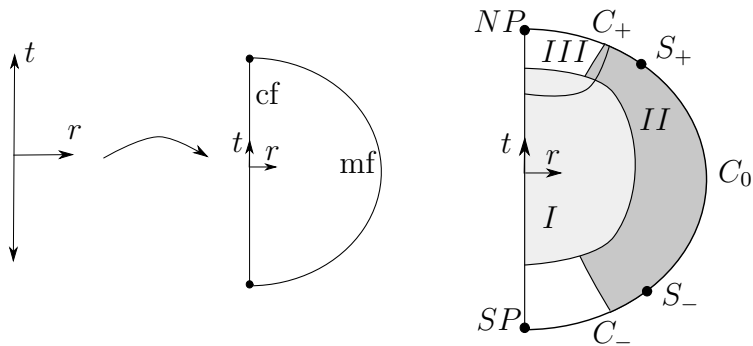
  \caption{The compactification $M$ of $\reals_t \times (0,\infty)_r \times Z$
    with the $Z$ factor suppressed, mapping on the left with detail on
    the right.  The image of $\reals_t \times (0,\infty)_r$ can be thought of either as the right half disk $D_+$ or the quarter sphere
  $\mathbb{S}_{++}$.  $S_\pm$ are, respectively, collapsed future/past
null infinity.  $C_\pm$ and $C_0$ are three subsets of $\mf$
corresponding to past and future timelike infinity and spacelike
infinity, respectively.}
  \label{fig:1-d-compact}
\end{figure}

\begin{figure}
  \centering
      \def\svgwidth{100mm}
   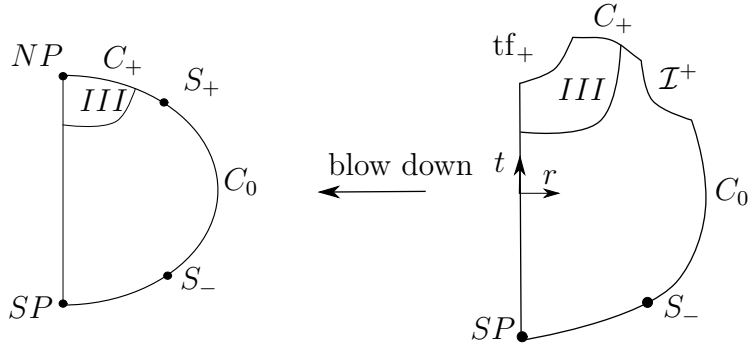
   \caption{A schematic view of the space $X$ (on the right) we use to
     study Price's Law.  The labeled region $R$ is the neighborhood of
     timelike infinity in which the solutions under consideration
     exhibit distinct asymptotic behavior at the three boundary
     hypersurfaces.}
  \label{fig:rf-blowup}
\end{figure}

To understand the detailed asymptotic behavior of solutions we perform
two blow ups of the compactification $M$. The first introduces future
null infinity, which is the natural domain of definition of the
radiation field, while the second blow up introduces the precise
region needed to distinguish between the limits of timelike paths with
zero and non-zero spatial momentum, thus providing two
  boundary hypersurfaces which distinguish the two regimes of
  asymptotic behavior of solutions at timelike infinity.

We recall from previous work~\cite{BVW1,BVW2} the first blow up, which
induces the construction of the manifold with corners on which the
radiation field naturally lives. As in those works, we
denote collapsed null infinity by $S$; this is the intersection of the
closure of maximally extended null geodesics with the introduced
boundary $\mf$.  The submanifold is given by $S = \{ v = \rho = 0\}$
where $v$ is a smooth function on $M$ which vanishes simply on the
sets $x = \pm 1$.  (Such $v$ should be chosen to be independent of
$\theta \in \sphere^{2}$ so that $v, \rho$ and
coordinates on the sphere form local coordinate charts near $S$.)
This submanifold naturally splits into two pieces according to whether
$\pm t > 0$ near the component, which we denote by $S_{\pm}$.  The
complement of $S$ in $\mf$ consists of three regions: the region
$C_{0}$ being those points in $\mf$ where $v < 0$, while the region in
$\mf$ where $v>0$ has two components, denoted $C_{\pm}$ according to
whether $\pm t > 0$ nearby.

We now \emph{blow up} $S_+$ in $M$ by replacing it with its inward
pointing spherical normal bundle.\footnote{The reader may wish to
  consult Melrose's book~\cite{Melrose:APS} for more details on the
  blow-up construction.}  In the product cone setting, this is
equivalent to blowing up a pair of points in $\sphere^{2}_{++}$ and
then taking the product with the link $Z$.  This process replaces $M$ with a
new manifold $\overline{M} = [M; S_+]$ on which polar coordinates around
the submanifold are smooth; the structure of this manifold with
corners depends only on the submanifold $S_+$ (and not on the particular
choice of defining functions $v$ and $\rho$).  

The space $\overline{M}$ -- which is not yet the final resolution we
consider, as we have yet to blow up the north pole (zero
momentum future timelike infinity) -- is again a manifold with corners and has
four boundary hypersurfaces: the lifts of $C_+$, the union $C_0 \cup C_-$, the lift of $\cf$, again denoted $\cf$, and the
new boundary hypersurface consisting of the pre-image of $S_+$ under
the blow-down map.  This new boundary
hypersurface, which is essentially future null infinity,
is denoted $\scri^{+}$.  Moreover, $\scri^{+}$ is
naturally a fiber bundle over $S_{+}$ with fibers diffeomorphic to
intervals.  Indeed, the interior of the fibers is naturally an affine
space (i.e., $\reals$ acts by translations, but there is no natural
origin).  Figure~\ref{fig:rf-blowup} depicts this blow-up
construction.  Given choices of functions $v$ and $\rho$, the fibers of the interior of
$\scri^{+}$ in $\overline{M}$ can be identified with $\reals \times Z$ via
the coordinate $s = v/\rho$.

Our focus on the $t > 0$ region of space-time is a matter of
convenience only; it is also possible, indeed very useful, to blow up
$S_-$ and thereby obtain a boundary hypersurface corresponding to past
null infinity.  Indeed, the face obtained by the blow up of $S_-$ is
the domain of the past radiation field.  For Dirac-type equations, one
can see distinct behavior (though related by a parity
  transformation) when comparing $\scri^+$ to the corresponding face
$\scri^-$ over $S_-$, but these details are irrelevant for our
discussion here which treats only the future timelike infinity
behavior.

Finally, we blow up the ``north pole'' ${\rm NP}= \{x = 0, \rho = 0 \}$.  We
  thereby obtain a manifold with corners $X = [\overline{M}; {\rm NP}]$ with
  an additional boundary hypersurface ${\rm tf}_+$ which separates
  $\cf$ from $C_+$.  Near the intersection of ${\rm tf}_+$ with $C_+$,
  we can use the boundary defining function $\rho_{{\rm tf}_+} = x +
  \rho$, while a full set of coordinates is given by
  $$
\rho_{{\rm tf}_+} = x+\rho, \quad \yyy =  \frac{x -
  \rho}{x + \rho}, \quad \theta,
$$
where $\theta$ are local coordinates on $\sphere^{2}$.

We now come to the main point of this blow up construction, which can
be understood as follows.  Still working on $\overline{M}$, the solutions to the differential equations
we consider here are tempered distributions which are conormal (on
$\overline{M}$) and admit an asymptotic expansion at $C_+$ with
polyhomogeneous coefficients.  In this
context, conormality\footnote{Here we choose to use $L^\infty$-based
  conormal spaces as opposed to the $H^s_b$-based spaces used in
  \cite{coneradfield} as the $L^\infty$-based spaces make the proof of
Lemma \ref{lemma:phg-lift} slightly clearer.  The equivalence of
$L^\infty$-based conormality and Sobolev space-based conormality
follows from the Sobolev embedding theorem and is discussed in
\cite{Melrose:MWC} in Chapter 4.} is equivalent to the existence of weights $\ell, \ell' \in
\mathbb{R}$ such that for all $j,  k \in \mathbb{N}_0, \alpha \in
\mathbb{N}_0^{\dim Z}$,
$$
(x \partial_x)^j (\rho \partial_\rho)^k \partial_\theta^\alpha
(x^{-\ell} \rho^{-\ell'} u)  \in L^\infty,
$$
while polyhomogeneity at $C_+$ is the condition that $u$ admits an
asymptotic expansion with terms of the form
$\rho^s (\log \rho)^k a(x, \theta)$ with $s \in \mathbb{C}$,
$k \in \mathbb{N}_0$ and $k$ bounded for $\Re s < L$.  The precise
nature of the expansion here is the key to our arguments.
We describe two notions of polyhomogeneity below.  First,
  we describe the more common notion, here termed (full)
  polyhomogeneity, which characterizes the joint expansion at
  $C_{+}\cap \scri^{+}$.  We then describe the slightly weaker notion
  governing the expansion at $C_{+}\cap \cf$, which we call partial
  polyhomogeneity. Both notions describe expansions in terms of
  \emph{index sets}; an index set $\E$ is a discrete subset of
  $\mathbb{C}\times \mathbb{N}_{0}$ so that $\{ (\sigma,k) \in \E :
  \Im \sigma > -R\}$ is finite for each $R$.

Recall that a distribution $w$ on
$\overline{M}$ is \textbf{(fully) polyhomogeneous} at $C_{+} \cap \scri^{+}$ if it admits a
joint asymptotic expansion there, meaning that $w$ is conormal and for
any $L$ 
$$
w(v, \zeta, \theta) = \sum_{(\sigma, k) \in \E, \Im \sigma > -L \atop
  (\tau, l) \in \mathcal{F}, \Im \tau> -L} v^{i\sigma}(\log
v)^{k} \zeta^{i\tau}(\log
\zeta)^{l}b_{\sigma k \tau l}(\theta) + (v\zeta)^{L} w'(v, \zeta, \theta),
$$
with $w'$ a bounded conormal distribution, together with separate
asymptotics at the two faces.  (Here $v$ is a defining function for
$\scri^{+}$ while $\zeta = \rho/v = s^{-1}$ is a defining function
for $C_{+}$.) 

The distinction between partially and
fully polyhomogeneous distributions lies in the fact that
partially polyhomogeneous distributions may not in general have an
expansion at $\cf$ at all, in particular they may not have a joint
expansion.  (Incidentally, it is possible to have expansions at the
interiors of both faces separately but no joint expansion, a simple
example being the function $r = \sqrt{x^2 + y^2}$ on the positive
upper half quadrant.)

A fully polyhomogeneous distribution on the whole of
  $\overline{M}$ admits expansions at all faces and joint expansions
  at all intersections of faces.  In contrast, the partially
  polyhomogeneous distributions which arise as solutions here admit
  expansions at all faces except possibly $\cf$.  Our analysis
  therefore focuses on the region near the intersection of $\cf$ with
  $C_+$.  Specifically, given index sets $\mathcal{E}$ and $\F$,
we say that a distribution $u$ is \textbf{partially
    polyhomogeneous} at $C_{+}$ (near $\cf$) with index set $\E$ at
  $C_{+}$ and $\F$ at $\cf$ if, for each $L$, in sets $\{ x < c \}$,
i.e.\ away from $\scri_+$,
\begin{equation}\label{eq:first expansion}
  u(x, \rho, \theta) = \sum_{(\sigma, k) \in \E, \Im \sigma > -L} \rho^{i\sigma}(\log
  \rho)^{k}a_{\sigma k}(x, \theta) + \rho^{L} u'(x, \rho, \theta),
\end{equation}
where the remainder $u'(x, \rho, \theta)$ is a bounded conormal
function and the $a_{\sigma k}$ are themselves polyhomogeneous
distributions on $C^+$ with a fixed index set $\mathcal{F}$ at
$\cf$.  (This means $a_{\sigma k} \sim \sum _{(\tau, l) \in
  \mathcal{F}} x^{i \tau} (\log x)^p c_{\tau, p}$ as $x \to 0$.)
In
other words, $u$ enjoys an expansion at $C_{+}$ in which each terms is
jointly polyhomogeneous at $C_{+}$ and $\cf$ but the remainder is in principle
only conormal at $\cf$.
\emph{Our main observation is that such a
distribution $u$ pulls back to the blown up space $X$ to be fully
polyhomogeneous at $C^+$ with index set $\mathcal{E}$ and partially
polyhomogeneous at $\tf_+$ with index set $\mathcal{F} +
\mathcal{E}$ at $\tf_+$ and index set $\mathcal{F}$ at $\cf$.}
This allows us to conclude the distinct asymptotic behavior described
above.  We formalize this pullback statement in Lemma \ref{lemma:phg-lift}.

As polyhomgeneity is essentially a local property near boundary faces
of a manifold with corners, we consider now the model of a manifold
with codimension two corners.  Let
$$
\reals_{++}^{2} = \{ (z, w) \in \mathbb{R}^2:  z, w \ge 0 \}
$$
denote the closed upper right quadrant, let $H_1 = \{ z = 0 \}, H_2 =
\{ w = 0 \}$.  Let $v \colon
\reals_{++}^{2} \longrightarrow \mathbb{C}$ be a conormal
distribution supported near $0$, and assume that $v$ is partially
polyhomogeneous at $H_1$ with index sets $\mathcal{E}$ at $H_1$ and
$\mathcal{F}$ at $H_{2}$. Let $\ff$ denote the new face of the blow up
$[\reals_{++}^{2} : (0,0) ]$.
\begin{lemma}
  \label{lemma:phg-lift}
  Let $U \subset \mathbb{R}^N$ be an open set. Let $v$ be a
  distribution on $\reals_{++}^{2} \times U$ which is partially
  polyhomogeneous at $H_1$ with index set $\mathcal{E}$ at $H_1$ and
  $\mathcal{F}$ at $H_2$.  
  Then the pullback of $v$ to the blown-up space
  $[\reals_{++}^{2};\{(0,0)\}] \times U$ is fully polyhomogeneous at
  $H_1$ with index set $\E$ and partially polyhomogeneous at $\ff
  \times U$ with index set $\E + \mathcal{F}$ there and index
  $\mathcal{F}$ at $H_2$.
  \end{lemma}

\begin{proof}  
  We suppress the factor $U$ in the domain as the smooth dependence on
  that parameter is irrelevant in the proof.

  By assumption, near $(0,0)$ we can write, for every $L, M \in
  \mathbb{R}$,
  \begin{equation}\label{eq:some expansions}
    \begin{split}
  v &= \sum_{(\sigma, k) \in \E, \Im \sigma > -L} z^{i\sigma}(\log
      z)^{k}a_{\sigma k} + z^{L} v' \\
    &= \sum_{(\sigma, k) \in \E, \Im \sigma > -L} z^{i\sigma}(\log
      z)^{k}\large( \sum_{(\tau, l) \in \mathcal{F}, \Im \tau > -M} w^{i\tau}
      (\log w)^l c_{\sigma k \tau l} + w^M v'' \large) + z^{L} v'
    \end{split}
  \end{equation}
  where $v', v''$ are conormal functions.

Blowing up, the functions $w, s = z/w$ define coordinates near the
intersection of $\ff$ with $H_1$.  The first line of \eqref{eq:some
  expansions}, with $z = sw$, continues to provide a polyhomogeneous
expansions with index set $\mathcal{E}$ but now with index set
$\mathcal{E} + \mathcal{F}$ at $\ff$ since
  \begin{align*}
  v = \sum_{(\sigma, k) \in \E, \Im \sigma > -L} s^{i\sigma} w^{i \sigma} (\log s +
    \log w)^k
 a_{\sigma k}(w) + (sw)^{L} v'.
\end{align*}
As pullbacks via blowdown maps preserve
conormality~\cite[Section 4.11]{Melrose:MWC}, $v'$ is conormal on
$[\reals_{++}^{2};\{(0,0)\}]$.

Near the intersection of $\ff$ with $H_2$, $z, t = w/z$ define
coordinates, with $z$ a defining function of $\ff$ and $t$ a defining
function of $H_2$, so using the second line of \eqref{eq:some expansions} we
get
$$
v= \sum_{(\sigma, k) \in \E, \Im \sigma > -L \atop (\tau, l) \in \mathcal{F}, \Im \tau> -M} z^{i\sigma + i \tau}(\log
z)^{k} t^{i\tau}(\log z + \log t)^l c_{\sigma k \tau l} + z^R v'''
$$
where $v'''$ is a conormal distribution and
$$
R < \inf\{ L , M - \Im \sigma - \Im \tau : (\sigma, 0) \in
\mathcal{E}, (\tau, 0) \in \mathcal{F} \}.
$$
This gives the asymptotic expansion with index set $\mathcal{E} +
\mathcal{F}$ in $z$ with terms having expansions in index set
$\mathcal{F}$ in $t$, and choosing $M, L$ sufficiently negative gives
the result.
\end{proof}

\section{Asymptotic expansions and Proof of Theorem \ref{metathm}}
\label{sec:an-asympt-expans}

That solutions of the equations considered here enjoy
  partial polyhomogeneity near $C_{+}$ and $\cf$ as well as (full)
  polyhomogeneity near $C_{+}\cap \scri^{+}$ follows from the prior works~\cite{DirCoulRF, coneradfield}.    Those papers describe the
asymptotics of the radiation field for this system, i.e., the behavior
of solutions near $C_{+}\cap \scri^{+}$.  In doing so, it
characterizes the asymptotic behavior all along $C_{+}$ in terms of
powers of $\rho$, $x$, and hypergeometric functions along $C_{+}$.

  We first turn our attention to the massless Dirac--Coulomb
  system~\cite{DirCoulRF}.  To that end, we define the relevant index
  sets\footnote{For the equations considered here, there are no
    logarithms in the expansions and so we drop the $\mathbb{N}_{0}$
    factor from the index set notation.} for
  $\abs{\charge}<1/2, \charge \neq 0$:
\begin{align*}
  \E_{\DC} &= \left\{ -i \left( 2 + \ell + \sqrt{\kappa
             ^{2}-\charge^{2}}\right) : \ell \in \mathbb{N}_{0},
             \kappa \in \mathbb{Z}\setminus \{ 0\}\right\}, \\
  \F_{\DC} &= \left\{ -i \left( -1 + j +
             \sqrt{\kappa^{2}-\charge^{2}}\right) : j \in  \mathbb{N}_{0}, \kappa \in \mathbb{Z}\setminus \{0\}\right\}.
\end{align*}

One consequence of that work is the following theorem:
\begin{theorem}
  \label{thm:dirac-coulomb}
  If $\psi$ is the solution to the massless Dirac--Coulomb
  system~\eqref{eq:dc} with smooth, compactly supported initial data,
  then, on $\overline{M}$, $\psi$ is (fully) polyhomogeneous at $C_{+}\cap \scri^{+}$ with
  index set $\E_{\DC}$ at $C_{+}$ and partially polyhomogeneous at
  $C_{+}$ near $\cf$ with index sets $\E_{\DC}$ at $C_{+}$ and
  $\F_{\DC}$ at $\cf$.  
\end{theorem}

In particular, from Lemma~\ref{lemma:phg-lift}, we conclude
that, on $X$, $\psi$ is partially polyhomgeneous
  at $\tf_{+} \cap \cf$ with index sets $\E_{\DC}+\F_{\DC}$ at $\tf_{+}$ and
  $\F_{\DC}$ at $\cf$.  It is additionally (fully) polyhomogeneous at
  all other boundary faces; at $C_{+}$ the index set is $\E_{\DC}$.
The bounds in Theorem~\ref{metathm} for solutions of the massless
Dirac--Coulomb system then follow by considering the largest term in
the expansions at $\tf_{+}$ and $C_{+}$, respectively.

\emph{This allows immediately for the proof of Theorem
  \ref{metathm} for Dirac--Coulomb, in the sense that the theorem
  statements are merely appropriate interpretations of the
  polyhomogeneity statements we have derived.  Indeed, a distribution $\psi$ which
is partially polyhomogeneous at $\tf_{+} \cap \cf$ with index sets
$\E_{\DC}+\F_{\DC}$ necessarily has an asymptotic
expansion with leading order $t^{-3 - \alpha(\charge)}$ at that face;
thus, in particular, if $U \Subset \tf_{+}$ is any open set compactly
contained in $\tf_{+}$ then: (1) the flat space coordinate $z$ is a
valid coordinate on $U$ and (2) $\psi(z, t) t^{3 + \alpha(\charge)}$
is a bounded function for $z \in U$.  The spatial cutoff $\chi$ in the
statement of Theorem \ref{metathm} serves exactly to cutoff to such a
set $U$ and thus for such $\chi$ we have $\chi(z) \psi(z, t) t^{3 +
  \alpha(\charge)}$ is bounded.  On the other hand, any ray $(t,
\gamma t, \theta)$ with $0 < |\theta| < 1$ approaches the interior of
the face $C^+$ as $t \to \infty$, the function $\psi(t,
\gamma t, \theta)$ is merely the restriction of $\psi$ to a particular
path of approach to this boundary hypersurface; thus the (full)
polyhomogeneity statement at that face immediately gives the stated
$\sim t^{-3 - \alpha(\charge)/2}$ behavior claimed in the theorem.}   

For the wave equation~\eqref{eq:inv-square} on
$\reals \times \reals^{n}$ with an inverse square potential, we appeal
to the results of the earlier paper about conic
manifolds~\cite{coneradfield}.  That paper describes the asymptotics
of solutions of the wave equation on cones, but the techniques apply
with no change to the wave equation on $\reals^{n}$ with an inverse
square potential.  To that end, we define the relevant
  index sets\footnote{As before, we drop the $\mathbb{N}_{0}$ factor
    from the notation.} for the inverse square (IS) problem with
  $\digamma > - \left( \frac{n-2}{2}\right)^{2}, \digamma \neq 0$.  In
  the below, we let
  $\nu_{j} = \sqrt{\left( \frac{n-2}{2}\right)^{2} +
    \lambda_{j}+\digamma}$, where $\lambda_{j}$, $j=0, 1, \dots$,
  denote the eigenvalues of the Laplacian, i.e.,
  $\lambda_{j} = j(j+n-2)$.
  \begin{align*}
    \E_{\invsq} &= \left\{ - i \left( \frac{n}{2} + k + \nu_{j}
                  \right) : j,k = 0, 1, 2, \dots,
                  \frac{1}{2}+\nu_{j}\notin \mathbb{Z} \right\}, \\
    \F_{\invsq} &= \left\{ -i\left( -\frac{n-2}{2} + \ell +
                  \nu_{j}\right) : j,\ell = 0, 1, 2, \dots, \nu_{j}
                  \notin \mathbb{Z}\right\}.
  \end{align*}
The methods of the paper~\cite{coneradfield} about product cones,
together with the calculation in Appendix~\ref{sec:reson-calc},
then establish the following theorem:
\begin{theorem}
  \label{thm:inv-square}
  If $u$ is the solution of the wave equation~\eqref{eq:inv-square}
  with smooth, compactly supported initial data, then, on
  $\overline{M}$, $u$ is (fully) polyhomogeneous at
  $C_{+}\cap \scri^{+}$ with index set $\E_{\invsq}$ at $C_{+}$ and
  partially polyhomogeneous at $C_{+}$ near $\cf$ with index sets
  $\E_{\invsq}$ at $C_{+}$ and $\F_{\invsq}$ at $\cf$.
\end{theorem}

In that paper, the terms in the expansion are characterized by the
resonances of the associated Laplacian on a hyperbolic cone.  If
$\frac{1}{2}+\nu_{j}$ is an integer, it is a pole of the scattering
matrix but not of the resolvent; in terms of the wave equations
studied here, this means that the associated term is supported
entirely in $S_{+}$ and vanishes in $C_{+}$ and therefore does not
contribute to the expansion at $C_{+}$ or $\tf_{+}$.  

Just as in the case of the Dirac--Coulomb problem,
Lemma~\ref{lemma:phg-lift} allows us to conclude that, after the
blowup, $u$ is partially polyhomogeneous at $\tf_{+}$ near $\cf$ with
index sets $\E_{\invsq} + \F_{\invsq}$ at $\tf_{+}$ and $\F_{\invsq}$
at $\cf$.  The remaining bounds in Theorem~\ref{metathm} again follow
by considering the largest terms in the expansions.  If $\frac{1}{2} +
\nu_{0}$ is
an integer, the leading term is instead 
$$\frac{n}{2}+\nu_{1} =
\frac{n}{2} + \sqrt{\left(
    \frac{n-2}{2}\right)^{2}+ n-1 + \digamma}$$ .
As $\digamma \neq 0$, $\frac{1}{2}+\nu_1$ cannot also be an
integer.\footnote{When $\frac{1}{2} + \nu_{0}$ is an integer, the
  radial part (i.e., the projection onto the lowest spherical
  harmonic) of the solution solves a conjugated odd-dimensional wave
  equation and therefore has no support at timelike
  infinity.}
The conclusion of the parts of Theorem \ref{metathm}
  related to inverse square potentials follow exactly as in the above
  discussion of Dirac--Coulomb.

\subsection{Sharpness}
\label{sec:sharp}

We now turn to a discussion of the sharpness of the
  bounds.  If the first term in each expansion is nonzero, the
  corresponding bound is saturated.  In the previous
  papers~\cite{DirCoulRF, coneradfield}, each term in the expansion
  arises via the inverse Mellin transform from a finite rank operator
  applied to the Mellin transform of the inhomogeneous data.  (For the
  initial value problem, one first converts the problem into an
  inhomogeneous forward problem.)  These operators are the residues of
  the poles of the inverse of a family of operators related to the
  equation.  

  To show that the first term is generically non-zero, it therefore
  suffices to characterize the first operator obtained this way.  As
  described in prior work, the first operator has rank one and is
  given by $\phi \otimes \psi$, where $\phi$ is in the kernel of the
  normal operator and is supported in $\overline{C}_{+}$, while $\psi$
  lies in the kernel of the adjoint of the normal operator and is
  regular at $S_{+}$.  Both $\phi$ and $\psi$ can be described in
  terms of hypergeometric functions; an energetic reader can follow
  the arguments of the prior work~\cite[Section 7]{DirCoulRF} to
  verify that $\psi$ is non-trivial in $C_{+}$.\footnote{In the case
    when
    $\frac{1}{2} + \sqrt{\left(\frac{n-2}{2}\right)^{2}+\digamma} \in
    \mathbb{Z}$, the hypergeometric function underlying $\phi$ is in
    fact a derivative of the delta function and so there is no
    contribution to the expansion at $C_{+}$.}  In particular, for
  most inhomogeneities (and hence most initial data), the first term
  in the expansion is a nonzero multiple of $\phi$.

  These arguments are similar in flavor to those of
  Hintz~\cite{Hintz_PL} demonstrating sharpness.  In both cases the
  leading term arises as the output of a rank one operator applied to
  the relevant data and thus is generically non-zero.

\appendix

\section{Resonance calculation}
\label{sec:reson-calc}

Let $\rho = \frac{1}{t+r}$ and $x = \frac{2r}{t+r}$ so that $r =
\frac{x}{2\rho}$.  If we set
\begin{align*}
  L_{0} &= \pd[t]^{2} - \lap + \frac{\digamma}{r^{2}}\\
  & = \pd[t]^{2} - \pd[r]^{2} - \frac{2}{r}\pd[r] -
    \frac{1}{r^{2}}\lap_{\sphere^{n-1}} + \frac{\digamma}{r^{2}},
\end{align*}
then $L_{0}$ lifts to
\begin{align*}
  L_{0} &= \rho^{2} (\rho\pd[\rho] + x\pd[x])^{2} + \rho^{2} (\rho
  \pd[\rho] + x\pd[x]) - \rho^{2} (2\pd[x] - x\pd[x] -
  \rho\pd[\rho])^{2} \\
  &\quad + \rho^{2} (2\pd[x] - x\pd[x] - \rho\pd[\rho]) -
  2\rho^{2}\frac{n-1}{x}(2\pd[x] - x\pd[x] - \rho \pd[\rho]) -
  \frac{4\rho^{2}}{x^{2}}\lap_{\sphere^{n-1}} + \frac{4\digamma\rho^{2}}{x^{2}}.
\end{align*}

The paper~\cite{coneradfield} and its predecessors consider the reduced normal
operator of 
\begin{align*}
  L &= \rho^{-2}\rho^{-\frac{n-1}{2}}L_{0}\rho^{\frac{n-1}{2}}.
\end{align*}
As this operator is homogeneous of degree $0$ in $\rho$, the reduced
normal operator is given by replacing $\rho \pd[\rho]$ with $i\sigma$,
yielding
\begin{align*}
  P_{\sigma} &= -\widehat{N}(L) = 4(1-x)\pd[x]^{2} + \frac{n-1}{x}
               4\pd[x] - \left( 4 + (i\sigma + \frac{n-1}{2}) +
               2(n-1)\right)\pd[x] \\
               & \hspace{1cm} +
               \frac{4}{x^{2}}\lap_{\sphere^{n-1}} - \frac{4\digamma}{x^{2}}
               - 2 \left( \frac{n-1}{x}\right)(i\sigma + \frac{n-1}{2}).
\end{align*}
The poles of the inverse of this operator (on appropriate
$\sigma$-dependent variable-order Sobolev spaces) yield the exponents
seen above.  The corresponding resonant states are solutions $v$ of
$P_{\sigma}v = 0$ lying in these same spaces.

For this particular operator, we can see the poles explicitly in terms
of the failure of hypergeometric functions to remain linearly
independent.  Indeed, separating into angular modes $\phi_{j}$ with
eigenvalues $-\lambda_{j}$, the radial coefficients $v_{j}$ of a solution of
$P_{\sigma}v_j = 0$ must satisfy
\begin{align*}
 0 &= xP_{\sigma,j}v_{j}  \\
 & = 4x (1-x)\pd[x]^{2}v_{j} + 4\left( n-1 -
  x(n+i\sigma)\right)\pd[x]v_{j} - 4\left(
  \frac{\digamma+\lambda_{j}}{x}\right)v_{j} - (n-1)(2i\sigma + n-1)v_{j} .
\end{align*}
Dividing by $4$, setting $P = \frac{1}{4}xP_{\sigma,j}$ and letting
$w=v_{j}$, $w$ must satisfy
\begin{align*}
  Pw &  = x(1-x)\pd[x]^{2} w + (n-1-x(n+i\sigma))\pd[x]w -
  \frac{\digamma+\lambda_{j}}{x}w - \left( \frac{n-1}{2}\right)(i\sigma+\frac{n-1}{2})w \\
 &  = 0.
\end{align*}

Conjugating by $x^{\alpha}$, where
\begin{equation*}
  \alpha = - \frac{n-2}{2} + \sqrt{\left( \frac{n-2}{2}\right)^{2} +
    \digamma+\lambda_{j}},
\end{equation*}
yields a hypergeometric equation for $y = x^{-\alpha}w$:
\begin{align*}
  x(1-x)\pd[x]^{2}y + (n-1 + 2\alpha - x (n+i\sigma + 2\alpha))\pd[x]y
  - \left( \alpha + \frac{n-1}{2}\right) \left( \alpha +
  \frac{n-1}{2}+i\sigma\right)y = 0.
\end{align*}
This is a hypergeometric differential equation with parameters (see,
e.g., \cite{DLMF} for notation)   
\begin{align*}
  a &= \frac{1}{2} + \sqrt{\left( \frac{n-2}{2}\right)^{2}
      +\lambda_{j} + \digamma},\\
  b &= \frac{1}{2} + i\sigma + \sqrt{\left( \frac{n-2}{2}\right)^{2}
      +\lambda_{j} + \digamma}, \\
  c &= 1 + 2 \sqrt{\left( \frac{n-2}{2}\right)^{2} +\lambda_{j} + \digamma}.
\end{align*}
The requirement that the solution $v$ lies locally in $H^{1}$ near
$x=0$ implies that $y$ must be a multiple of the hypergeometric function
\begin{equation*}
  y_{1} = F(a, b, c; x) = F\left( \frac{1}{2} + s, \frac{1}{2}+i\sigma
    + s, 1 + 2s; x\right),
\end{equation*}
where $s = \sqrt{\left( \frac{n-2}{2}\right)^{2}+\lambda_{j} + \digamma}$.
On the other hand, if a solution is to lie in the appropriate variable
order Sobolev space globally, at $C_{+}$ it must be a multiple of
\begin{align*}
  y_{4} & = (1-x)^{c-a-b} F(c-a, c-b, c-a-b+1; 1-x)  \\
  & =
  (1-x)^{-i\sigma}F\left( \frac{1}{2} + s, \frac{1}{2} - i\sigma + s,
    1-i\sigma; 1-x\right).
\end{align*}
The poles under consideration are therefore given by those $\sigma$
for which $y_{4}$ is a multiple of $y_{1}$.  Kummer's connection
formulae~\cite[15.10.18]{DLMF} allow us to write $y_{4}$ in terms of
the basis  of solutions $y_{1}$ and $y_{2}$, where $y_{2}$ is given by
\begin{equation*}
  y_{2} = x^{1-c} F(a-c+1, b-c+1, 2-c; x) = x^{-2s}F\left( \frac{1}{2}
    - s, \frac{1}{2} - s + i\sigma, 1 - 2s; x\right).
\end{equation*}
In this case, we have
\begin{align*}
  y_{4} &= \frac{\Gamma (1-c)\Gamma(c-a-b+1)}{\Gamma
    (1-a)\Gamma(1-b)}y_{1} +
  \frac{\Gamma(c-1)\Gamma(c-a-b+1)}{\Gamma(c-a)\Gamma(c-b)}y_{2} .
\end{align*}
In particular, the coefficient of $y_{2}$ vanishes precisely when
$\frac{1}{2}+s-i\sigma$ is a pole of the gamma function, i.e., when
\begin{equation*}
  \sigma = \sigma_{j,k} = - i\left( \frac{1}{2} + s + k\right), \quad k = 0, 1, \dots .
\end{equation*}
Moreover, the resonant state $v_{j}$ associated with $\sigma_{j,k}$ is
a multiple of $x^{\alpha}y_{1}$ and therefore is polyhomogeneous with
leading order behavior
\begin{equation*}
  x^{-\frac{n-2}{2} + s} \text{ as }x \to 0.
\end{equation*}

\newcommand{\etalchar}[1]{$^{#1}$}

\end{document}